\documentclass[dvipdfm]{article}
\usepackage[T1]{fontenc}
\usepackage[ansinew]{inputenc}
\usepackage[dvips]{graphicx}
\usepackage{bbm}
\usepackage{amsfonts}
\usepackage{amssymb}
\newcommand{\lb}[1]{\label{#1}}

\title{Bifurcation delay - the case of the sequence:\\stable focus - unstable focus - unstable node}

\author{Eric Beno\^{\i}t\thanks{Laboratoire de Math\'ematiques et
    Applications, Universit\'e de la Rochelle, avenue Michel
    Cr\'epeau, 17042 LA ROCHELLE and/or projet COMORE, INRIA, 2004 route des Lucioles 06902 SOPHIA-ANTIPOLIS    courriel: ebenoit@univ-lr.fr}}

\date{\today}

\begin{document}
\newtheorem{theorem}{Theorem}
\newtheorem{proposition}[theorem]{Proposition}
\newtheorem{definition}{Definition}
\newtheorem{conjecture}[theorem]{Conjecture}
\newtheorem{lemma}[theorem]{Lemma}
\newcommand{\qed}{\ifmmode$\blacksquare$\else{\unskip\nobreak\hfil
	\penalty50\hskip1em\null\nobreak\hfil$\blacksquare$
	\parfillskip=0pt\finalhyphendemerits=0\endgraf}\fi}
\newtheorem{demth}{Proof}
	\renewcommand{\thedemth}{}
	\newenvironment{proof}{\begin{demth}\rm}{\qed\end{demth}}
\newcommand{\ip}{\mbox{\rm O\hspace{-0.67em}/}\hspace{0.3em}}
\newcommand{\eps}{\varepsilon}
\newcommand{\CC}{\mathbbm{C}}
\newcommand{\ud}{{\mbox{$\displaystyle\scriptstyle\frac{1}{2}$}}}
\newcommand{\ut}{{\mbox{$\displaystyle\scriptstyle\frac{1}{3}$}}}
\newcommand{\dt}{{\mbox{$\displaystyle\scriptstyle\frac{2}{3}$}}}
\newcommand{\td}{{\mbox{$\displaystyle\scriptstyle\frac{3}{2}$}}}
\newcommand{\uq}{{\mbox{$\displaystyle\scriptstyle\frac{1}{4}$}}}
\newcommand{\us}{{\mbox{$\displaystyle\scriptstyle\frac{1}{6}$}}}
\newcommand{\uud}{{\mbox{$\displaystyle\scriptstyle\frac{1}{12}$}}}
\newcommand{\systemededeux}[4]{\left\{\begin{array}{rcl}#1&=&#2\\#3&=&#4
	\end{array}\right.}
\newcommand{\systemedeux}[5]{\begin{equation}\lb{#1}\left\{
	\begin{array}{rcl}#2& = & #3 \\ #4 & = & #5 \end{array}\right.
	\end{equation}}
\maketitle

\begin{abstract}
Let us give a two dimensional family of real vector fields. We suppose that there exists a stationary point where the linearized vector field has successively a stable focus, an unstable focus and an unstable node. When the parameter moves slowly, a bifurcation delay appears due to the Hopf bifurcation. The studied question in this article is the continuation of the delay after the focus-node bifurcation.

AMS classification: 34D15, 34E15, 34E18, 34E20, 34M60

Keywords: Hopf-bifurcation, bifurcation-delay, slow-fast, canard, Airy, relief.

\end{abstract}

\section{Introduction}
\setcounter{footnote}{0}
"Singular perturbations" is a studied domain from many years ago. Since 1980, many contributions were written because new tools were applied to the subject. The main studied objects are the {\it slow fast vector fields} also known as {\it systems with two time-scales}. We will give the problem here with a more particular point of view: the {\it bifurcation delay} , as in articles \cite{Lob1,Ben12,Lobry92b,FruSch08}. We write the studied system:  $\eps\dot{X}=f(t,X,\eps)$, where $\eps$ is a real positive parameter which tends to zero. For a better understanding of the expression {\it dynamic bifurcation} it is better to write the system after a rescaling of the variable:
$$\systemededeux{\dot{X}}{f(a,X,\eps)}{\dot{a}}{\eps}$$
where $a$ is a "slowly varying" parameter.

The main objects in this study are the eigenvalues of the linear part of equation $\dot{X}=f(a,X,0)$ near the quasi-stationary point. Indeed, they give a characterization of the stability of the equilibrium of the fast vector field at this point. The aim of this study is to understand what happens when the stability of a quasi-stationary point changes. A bifurcation occurs when at least one of the eigenvalues has a null real part.   

In this article we restrict our study to two-dimensional real systems. In this situation, the generic bifurcations are: the saddle-node bifurcation, the Hopf bifurcation and the focus-node bifurcation.

The saddle-node bifurcation is solved by the {\it turning point} theory: when the real part of one of the eigenvalue becomes positive, there is no delay and a trajectory of the systems leaves the neighborhood of the quasi-stationary point when it reaches the bifurcation. For this study, the study of one-dimensional systems is sufficient: we have a decomposition of the phase space where only the one-dimensional factor is interesting. There exist many articles on this subject, we will be interested particularly by  \cite{BFSW97} where the method of {\it relief} is used. The article \cite{Fenichel79} introduces the geometrical methods of {\it Fenichel's manifold}. 

The Hopf delayed bifurcation is  well explained in  \cite{Wal2}, we will upgrade the results in paragraph \ref{para:H} below.

In a focus-node bifurcation, the stability of the quasi stationary point does not change, then, locally, there is no problem of canards or bifurcation delays. Indeed,  when there is a bifurcation delay at a Hopf-bifurcation point, it is possible to evaluate the value of the delay, and the main question is to understand the influence of the focus node bifurcation to this delay.

In paragraph \ref{para:H}, the Hopf bifurcation alone is studied, as well as   the focus-node bifurcation following  a Hopf bifurcation in paragraph \ref{para:HFN}.

In paragraphs \ref{par:Hc} and \ref{HFN-gdcan}, we assume that there exists a solution of the system approximed by the quasi steady state in the whole domain, so this trajectory has an infinite delay. The used methods are real, and the system has to be smooth (actually only $C^2$).
In paragraphs \ref{Hopf-butée} et \ref{para:HFNb}, we avoid this very special hypothesis. It is here supposed that the system is analytic, and we study the solutions on complex domains. 
Unfortunately, I have not a proof for the main result of this article. But it seems to me that the problem is interesting, and the results are argumented.

We use Nelson's nonstandard terminology (see for example \cite{DfR}). Indeed, almost all sentences can be translated in classical terms, where $\eps$ is considered as a variable and not as a parameter. Often, the translation is given on footnotes.

\section{The delayed Hopf bifurcation }\lb{para:H}
\setcounter{footnote}{0}

The problem is studied and essentially resolved in  \cite{Wal2}. We give here the proofs to improve the results and to fix the ideas for the main paragraph of the article. The main tool is the {\it relief'}s theory of J.L. Callot, explained in \cite{Cal4}.

The studied equation is
\begin{equation}\lb{eqH}\eps\dot{X}=f(t,X,\eps)\end{equation}
where $f$ is analytic on a domain $\cal D$ of $\CC\times\CC^2\times\CC$. 

\paragraph {Hypothesis and notations}
\begin{enumerate}
\item[H1]The function $f$ is analytic. It takes real values when the arguments are real.

\item[H2]The parameter $\eps$ is real, positive, infinitesimal\footnote{In classical terms, we assume that $\eps$ leaves in a small complex sector: $|\eps|$ bounded and $\arg(\eps)\in]-\delta,\delta[$.}.   

\item[H3]There exists an analytic function $\phi$, defined on a complex domain ${\cal D}_t$ 
so that $f(t,\phi(t),0)=0$. The curve $X=\phi(t)$ is called the {\it slow curve} of equation (\ref{eqH}). We assume that the intersection of ${\cal D}_t$ with the real axis is an interval  $]t_m,t_M[$.

\item[H4] Let us denote $\lambda(t)$ and $\mu(t)$ for the eigenvalues of the jacobian matrix $D_Xf$, computed at point $(t,\phi(t),0)$. We assume that , for $t$ real, the signs of the real and imaginary parts are given by the table below~:

\begin{tabular}{c|ccccc}
$t$ &$t_m$&~~~~~~~~~~& a &~~~~~~~~~&$t_M$\\ \hline
$\Re(\lambda(t))    $& & -      & 0 & +      &\\
$\Re(\mu(t))        $& & -      & 0 & +      &\\
$\Im(\lambda(t))    $& & -      & - & -      &\\
$\Im(\mu(t))        $& & +      & + & +     & 
\end{tabular}

Then, when $t$ increases from $t_m$ to $t_M$, the quasi-steady state is first an attractive focus, then a repulsive focus, with a Hopf bifurcation at $t=a$. 

\end{enumerate}
\subsection{Input-output function when there exists a big canard}\lb{par:Hc}

In this section, we assume that there exists a {\it big canard}  $\tilde{X}(t)$ i.e. a solution of equation (\ref{eqH}) such that\footnote{Without nonstandard terminology, a {\it big canard} is a solution of equation (\ref{eqH}) depending on the parameter $\eps$ such that 
$$\forall t\in]t_m,t_M[, \lim_{\eps>0,\eps\rightarrow 0}\tilde{X}(t,\eps)=\phi(t)$$} $\tilde{X}(t)\simeq \phi(t)$ for all $t$ in the $S$-interior of $]t_m,t_M[$. 
We now want to study the others solutions of equation (\ref{eqH}) by comparison with $\tilde{X}$. 

The main tool for that is a sequence of change of unknowm: first, we perform a translation on $X$, depending on $t$ to put the big canard on the axis:
$$X\ =\ \tilde{X}(t)\ +\ Y$$
It gives the system
$$\eps\dot{Y}\ =\ g(t,Y,\eps)\mbox{~~~~with~~~~}
g(t,Y,\eps)=f(t,\tilde{X}(t)+Y,\eps)-f(t,\tilde{X}(t),\eps)$$
The matrix $D_Xf(t,\phi(t),0)$ has two complex conjugate distinct eigenvalues (see hypothesis H4), then there exists a linear transformation $P(t)$ which transforms the jacobian matrix in a canonical form. We define the change of unknown 
$$Y=P(t)Z$$ 
The new system, has the following form (we wrote only the interesting terms):
$$\eps\dot{Z}\ =\ h(t,Z,\eps)\mbox{~~~~with~~~~}
h(t,Z,\eps)=
\left(\begin{array}{cc}\alpha(t)&-\omega(t)\\\omega(t)&\alpha(t)
\end{array}\right) Z+O(\eps)Z+O(Z^2)\mbox{~~~~,~~~~}\lambda(t)=\alpha(t)-i\omega(t)$$
The next change is given by the polar coordinates:
$$Z\ =\ \left(\begin{array}{c}r\cos\theta\\r\sin\theta
\end{array}\right)$$
$$\systemededeux{\eps\dot{r}}{r\left(\alpha(t)+O(\eps)+O(r)\right)}{\eps\dot{\theta}}{\omega(t)+O(\eps)+O(r)}$$
The last one is an exponential microscope\footnote{All the preceeding transformations were regular with respect to $\eps$. This last one is singular at $\eps=0$.}:
$$r\ =\ \exp\left(\frac{\rho}{\eps}\right)$$
\systemedeux{eqHr}{\dot{\rho}}{\alpha(t)+O(\eps)+e^{\frac{\rho}{\eps}}k_1(r,\theta,\eps)}{\eps\dot{\theta}}{\omega(t)+O(\eps)+e^{\frac{\rho}{\eps}}k_2(r,\theta,\eps)}
While $\rho$ is non positive and non infinitesimal, $r$ is exponentially small and the equation (\ref{eqHr}) gives a good approximation of $\rho$ with $\dot{\rho}=\alpha$. When $\rho$ becomes infinitesimal, with a more subtle argument (see \cite{Ben10}) using differential inequations, we can prove that $r$ becomes non infinitesimal. This gives the proposition below:

\begin{proposition}
Let us assume hypothesis H1 to H4 (Hopf bifurcation) for equation (\ref{eqH}). However, we assume that there exists a canard $\tilde{X}(t)$ going along\footnote{A solution $\tilde{X}(t,\eps)$ goes along the slow curve at least on $]t_1,t_2[$ if $$\forall t\in]t_1,t_2[, \lim_{\eps>0,\eps\rightarrow 0}\tilde{X}(t,\eps)=\phi(t)$$} the slow curve at least on $]t_m,t_M[$. Then if $X(t)$ goes along the slow curve exactly\footnote{A solution $\tilde{X}(t,\eps)$ goes along the slow curve exactly on $]t_e,t_s[$ if it goes along the slow curve at least on $]t_e,t_s[$, and if the interval $]t_e,t_s[$ is maximal for this property.} on $]t_e,t_s[$ with $[t_e,t_s]\subset]t_m,t_M[$, then 
$$\int_{t_e}^{t_s}\Re(\lambda(\tau))d\tau\ =\ 0$$

\end{proposition}
The input-output relation (between $t_e$ and $t_s$) is defined by $\int_{t_e}^{t_s}\Re(\lambda(\tau))d\tau\ =\ 0$. It is described by its graph (see figure \ref{fig:relHc}). In this case, this relation is a function.
\begin{figure}[ht]
\begin{center}
\includegraphics[height=3cm,width=5cm]{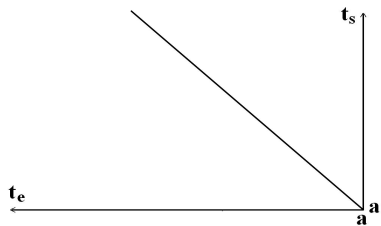}
\end{center}
\caption{The input-output relation for equation (\ref{eqHx}) when there exists a big canard.}\lb{fig:relHc}
\end{figure}

\subsection{The bump and the anti-bump}\lb{Hopf-butée}
\setcounter{footnote}{0}

In this paragraph, $t$ becomes complex, in the domain ${\cal D}_t$. We assume that for all $t$ in ${\cal D}_t$, the two eigenvalues $\lambda(t)$ and $\mu(t)$ are distinct. It is a necessary condition to apply Callot's theory of reliefs. 

We define the reliefs $R_\lambda$ and $R_\mu$ by:

$$F_\lambda(t)\ =\ \int_a^t\lambda(t)dt\mbox{~~~~~~,~~~~~~}
R_\lambda(t)=\Re(F_\lambda(t))$$
$$F_\mu(t)\ =\ \int_a^t\mu(t)dt\mbox{~~~~~~,~~~~~~}
R_\mu(t)=\Re(F_\mu(t))$$

It is easy to see that $\lambda(\overline{t})=\overline{\mu(t)}$, and  $F_\lambda(\overline{t})=\overline{F_\mu(t)}$, then $R_\lambda(\overline{t})\ =\ R_\mu(t)$.
The two functions $R_\lambda$ and $R_\mu$ coincide on the real axis. We will denote $R(t)$.

\begin{definition} \lb{def:chemindescendant}
We say that a path $\gamma: s\in[0,1]\mapsto {\cal D}_t$ goes down the relief  $R_\lambda$ if and only if  $\frac{d}{ds}R_\lambda(\gamma(s))<0$ for all $s$ in $[0,1]$.  
\end{definition}

\begin{definition} \lb{def:domainedescendant}
Let us give a point $t_e$ such that $(t_e,\phi(t_e),0)\in{\cal D}$.  We say that ${\cal D}_t$ is a domain {\em  below}  $t_e$ if and only if for all $t$ in the  $S$-interior of ${\cal D}_t$, there exist two paths in ${\cal D}_t$, from $t_e$ to $t$, the first one goes down the relief $R_\lambda$ and the second one down $R_\mu$. 
\end{definition}

\begin{theorem}[Callot]\lb{th:Callot} Let us assume that ${\cal D}_t$ is a domain below $t_e$. A solution $X(t)$ of equation (\ref{eqH}) with an initial condition $X(t_e)$ infinitesimally close to $\phi(t_e)$ is defined at least on the  $S$-interior of ${\cal D}_t$ where it is infinitesimally close to $\phi(t)$.
\end{theorem}

Let us apply this theorem to the following example, chosen as the typical example satisfying hypothesis H1 to H4 (Hopf bifurcation).
\systemedeux{eqHx}{\eps\dot{x}}{tx+y+\eps c_1}{\eps\dot{y}}{-x+ty+\eps c_2}
The eigenvalues are $\lambda=t-i$ et $\mu=t+i$. The level curves of the two reliefs $R_\lambda(t)=\ud (t-i)^2$ and $R_\mu(t)=\ud (t+i)^2$ are drawn on figure \ref{fig:reliefH}.

\begin{figure}[ht]
\begin{center}
\includegraphics[height=4cm,width=7cm]{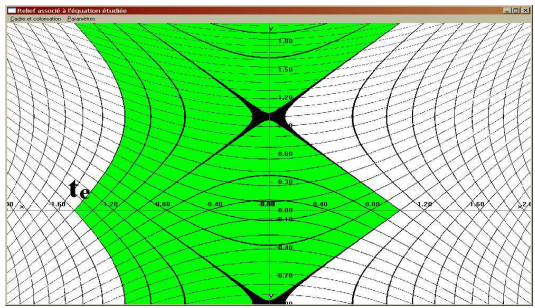}
\end{center}
\caption{The level curves of the two reliefs of equation (\ref{eqHx}), and a domain below $t_e$}\lb{fig:reliefH}
\end{figure}

Generically\footnote{I do not know the exact generic hypothesis. We have to combine the constraints given by the surstability theory of \cite{BFSW97} and the fact that the equation (\ref{eqH}) is real} there is no surstability at point $t=i$  (see \cite{BFSW97} for the definition of surstability). Consequently, we have the following results for all equations such that the  reliefs are on the same type as those of figure  \ref{fig:reliefH}:

\begin{definition}\lb{def:butee}
Let us give $t_c$ a point\footnote{In some cases, it is possible that $t_c$ is infinite. For   $\eps\dot{X}=
\left(\begin{array}{cc}\sin t&\cos t\\-\cos t&\sin t
\end{array}\right) X+O(X^2)+O(\eps)$ we have $t_c=+i\infty$.
} where the eigenvalue vanishes: $\lambda(t_c)=0$. The value of the relief at point $t_c$ is a critical value of the relief $R_\lambda$.  The {\em bump}\footnote{The name "bump" is a translation of the french name "but\'ee"
} is the real number $t^*$ bigger than  $a$, minimal such that $R_\lambda(t^*)$ is a critical value.  The {\em anti-bump} is the real number $t^{**}$ smaller than $a$, maximal such that $R_\lambda(t^{**})$ is a critical value.
\end{definition}

For equation (\ref{eqHx}), the bump is $t^*=1$ and the anti-bump $t^{**}=-1$.

\begin{theorem} A trajectory of equation (\ref{eqH}) can go along the slow curve  $X=\phi(t)$ exactly on $]t_e,t_s[$ if and only if one of the following is verified:
 $$t_e<t^{**}\mbox{~~~~and~~~~}t_s=t^*$$
 $$t_e=t^{**}\mbox{~~~~and~~~~}t_s>t^*$$
 $$t^{**}<t_e<a\mbox{~~~~and~~~~}a<t_s<t^*\mbox{~~~~and~~~~}R(t_e)=R(t_s)$$
\end{theorem}

This theorem is illustrated by the graph of the input-output relation, drawn on figure \ref{fig:relH}.

\begin{figure}[ht]
\begin{center}
\includegraphics[height=3cm,width=5cm]{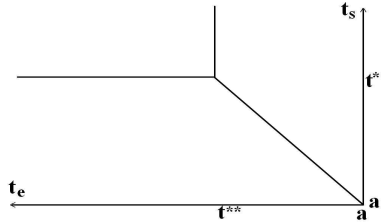}
\end{center}
\caption{The input-output  relation for equation (\ref{eqHx})}\lb{fig:relH}
\end{figure}

\section{Delayed Hopf bifurcation followed by a focus-node bifurcation}\lb{para:HFN}
\setcounter{footnote}{0}
The studied equation is
\begin{equation}\lb{eqX}\eps\dot{X}=f(t,X,\eps)\end{equation}
where $f$ is analytic on a domain $\cal D$ of $\CC\times\CC^2\times\CC$, and satisfies the following hypothesis:

\paragraph {Hypothesis and notations}
\begin{enumerate}
\item [HFN1] The analytic function $f$ takes real values when the arguments are real.

\item [HFN2] The parameter $\eps$ is real, positive, infinitesimal.   

\item [HFN3] There exists an analytic function $\phi$, defined on a complex domain ${\cal D}_t$ 
such that $f(t,\phi(t),0)=0$. The curve $X=\phi(t)$ is called the {\it slow curve} of equation \ref{eqH}. We assume that the intersection of ${\cal D}_t$ with the real axis is an interval  $]t_m,t_M[$.

\item [HFN4] Let us denote $\lambda(t)$ and $\mu(t)$ for the eigenvalues of the jacobian matrix $D_Xf$, computed at point $(t,\phi(t),0)$. We assume that , for $t$ real, the signs of the real and imaginary parts are given by the table below~:

\begin{tabular}{c|ccccc}
$t$ &$t_m$~~~~~~~~~~& a &~~~~~~~~~& b &~~~~~~~~~~$t_M$ \\ \hline
$\Re(\lambda(t))    $ & -      & 0 & +     & + & + \\
$\Re(\mu(t))        $ & -      & 0 & +     & + & + \\
$\Im(\lambda(t))    $ & -      & - & -     & 0 & 0 \\
$\Im(\mu(t))        $ & +      & + & +     & 0 & 0 
\end{tabular}

Then, when $t$ increases on the real interval $]t_m,t_M[$, we have succesively an attractive focus, a Hopf bifurcation at $t=a$, a repulsive focus, a focus-node bifurcation at $t=b$ and a repulsive node.
At point $t=b$, the two eigenvalues coincide. We assume that $\lambda(t)=\mu(t)$ only at point $b$.  Actually, in the complex plane, the two eigenvalues are the two determinations of a multiform function defined on a Riemann surface with a square root singularity at point $b$.

However, there is a symmetry: if the function  $\sqrt{~}$ is defined with a cut-off on the positive real axis, it satisfies $\overline{\sqrt{s}}=-\sqrt{\overline{s}}$ and we then have
$$\mu(\overline{t})\ =\ \overline{\lambda(t)}$$

\item [HFN5] For the same reason, the two reliefs 
$$R_\lambda(t)=\Re\left(\int_a^t\lambda(t)dt\right) \mbox{~~~~~~and~~~~~~} R_\mu(t)=\Re\left(\int_a^t\mu(t)dt\right)$$ are the two determinations of a multiform function with a square root singularity at point $t=b$. However, there is a symmetry: if the function  $\sqrt{~}$ is defined with a cut-off on the positive real axis, it satisfies $\overline{\sqrt{s}}=-\sqrt{\overline{s}}$ and we have then: $R_\mu(t)\ =\ R_\lambda(\overline{t})$ except on the cut-off half line $[b,+\infty[$. 
For real $t>b$, we choose determinations of square root such that $\lambda(t)<\mu(t)$. We assume that $R_\lambda$ has a unique critical point with critical value $R_c$. We assume that $R_\lambda(b)<R_c$. An example is given and studied in paragraph \ref{par:relief}. 
\end{enumerate}
\subsection{Input-output function when there exists a big canard}\lb{HFN-gdcan}
We assume now that there exists a {\it big canard}  $\tilde{X}(t)$ i.e. a solution of equation (\ref{eqH}) such that  $\tilde{X}(t)\simeq \phi(t)$ for all $t$ in the $S$-interior of $]t_m,t_M[$. 
The study below is similar to paragraph \ref{par:Hc}. The added difficulty is the coincidence of the two eigenvalues at point $b$ which do not allow to diagonalize the linear part. 

The first change of unknown is  $X=\tilde{X}(t)+Z$ which moves the big canard on the axis $X=0$:
$$\eps\dot{Z}\ =\ 
A(t) Z+O(\eps)Z+O(Z^2)\mbox{~~~~,~~~~}A(t)=D_Xf(t,\phi(t),0)$$
Let us denote  $\left(\begin{array}{cc}\alpha(t)&\beta(t)\\\gamma(t)&\delta(t)
\end{array}\right)$ the coefficients of the matrix $A(t)$. As in paragraph  \ref{par:Hc}, the change of unknowns 
$$Z\ =\ \left(\begin{array}{c}r\cos\theta\\r\sin\theta
\end{array}\right)\mbox{~~~~~~,~~~~~~}
r\ =\ \exp\left(\frac{\rho}{\eps}\right)$$
gives the new system:
\systemedeux{eqr}{\dot{\rho}}{\alpha(t)\cos^2\theta+(\beta(t)+\gamma(t))\cos\theta\sin\theta+\delta(t)\sin^2\theta+O(\eps)+e^{\frac{\rho}{\eps}}k_1(r,\theta,\eps)}
{\eps\dot{\theta}}{\gamma(t)\cos^2\theta+(\delta(t)-\alpha(t))\cos\theta\sin\theta-\beta(t)\sin^2\theta+O(\eps)+e^{\frac{\rho}{\eps}}k_2(r,\theta,\eps)}
For nonpositive $\rho$ (more precisely, for infinitesimal $r$), the second equation is a slow-fast equation. Its slow curve is given by
$$\theta\ =\ \arctan \left( {\frac {\delta(t)-\alpha(t)\pm\sqrt {{\alpha(t)}^{2}-2\,\alpha(t)\delta(t)+{\delta(t)}^{2}+4\,\beta(t)\gamma(t)}}
{2\beta(t)}} \right) $$
It has two branches when $\lambda$ and $\mu$ are reals, one is attractive, the other is repulsive: see figure \ref{fig:rt}.
\begin{figure}[htbp]
\begin{center}
\includegraphics[height=7cm,width=7cm]{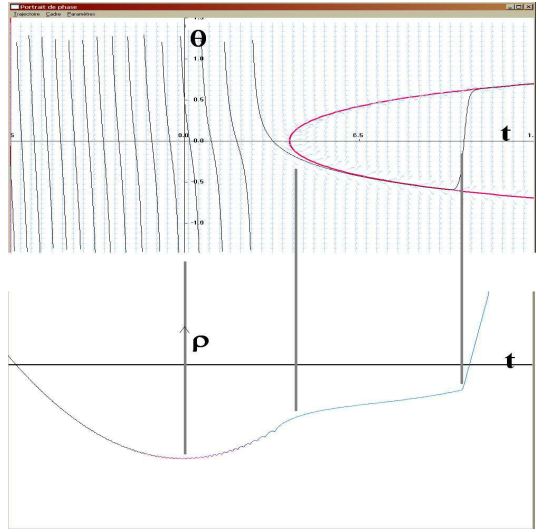}
\end{center}
\caption[]{One of the trajectories of system 
$0.002\dot{X}=\left(\begin{array}{cc}t&1\\t-0.3&t \end{array}\right)X$
drawn with the variables $(\theta,\rho)$. The slow curve is also drawn}\lb{fig:rt}
\end{figure}

When $\theta$ goes along a branch of the slow curve, (and when $r$ is infinitesimal), an easy computation shows that $\dot{\rho}$ is infinitely close to one of the eigenvalues $\lambda$ or $\mu$. The repulsive branch corresponds to the smallest eigenvalue (which is real positive). When $t<b$, the angle $\theta$ moves infinitely fast, and an averaging procedure is needed to evaluate the variation of $\rho$:
$$\langle\dot{\rho}\rangle\ =\ \frac{\int_{\theta_1}^{\theta_1+2\pi}\frac{\dot{\rho}}{\dot{\theta}}d\theta}
{\int_{\theta_1}^{\theta_1+2\pi}\frac{1}{\dot{\theta}}d\theta}$$
An easy computation shows now that, in the $S$-interior of the domain  $t<b$, $\rho<0$, we have
$$\langle\dot{\rho}\rangle\ \simeq\ \frac{\alpha(t)+\delta(t)}{2}\ =\ \Re(\lambda(t))=\Re(\mu(t))$$

Let us give an initial condition $(t,\theta)$ between the two branches of the slow curve and $\rho$ negative non infinitesimal (in the example, we can take $t=0.8$, $\theta=0$, $\rho=-0.03$). For increasing $t$, the curve $(t,\theta(t))$ goes along the attractive branch of the slow curve, while $\rho$ believes negative non infinitesimal. For decreasing $t$, the solution goes along the repulsive branch, then $\theta$ moves infinitely fast while $\rho$ believes negative non infinitesimal. Consequently, we know the variation of $\rho(t)$ (see figure \ref{fig:rt}). As in paragraph \ref{par:Hc}, a  more subtle argument is needed to prove that when $\rho$ becomes infinitesimal, the variable $r$ becomes non infinitesimal and the trajectory $X$ leaves the neighborhood of the slow curve.

From this study, all the behaviours of $\rho(t)$ are known, depending on the initial condition. They are drawn on figure \ref{fig:rhobis}.

\begin{figure}[ht]
\begin{center}
\includegraphics[height=4cm,width=8cm]{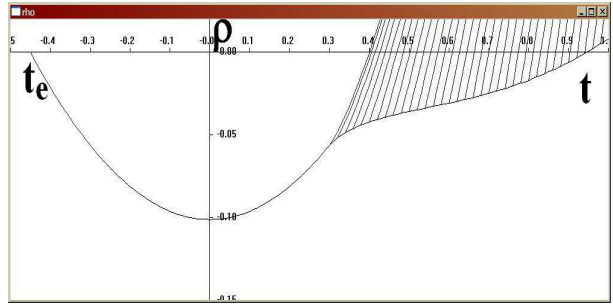}
\end{center}
\caption[]{The possible behaviours of $\rho(t)$.}\lb{fig:rhobis}
\end{figure}

\begin{proposition}\lb{prop:HFNcan}
Let us give an equation of type (\ref{eqx}) with hypothesis HFN1 to HFN5. Assume also that there exists a big canard $\tilde{X}(t)$ going along the slow curve on the whole interval $]t_m,t_M[$.
If a trajectory $X(t)$ goes along the slow curve exactly on an interval $]t_e,t_s[$ with $[t_e,t_s]\subset ]t_m,t_M[$, then
$$\int_{t_e}^{t_s}\Re(\lambda(\tau))d\tau\ \leq\ 0\ \leq\ \int_{t_e}^{t_s}\Re(\mu(\tau))d\tau$$

Conversely, if the inequalities above are satisfied, there exists a trajectory going along the slow curve exactly on $]t_e,t_s[$.
\end{proposition}

The input-output relation is described by its graph, drawn on figure \ref{fig:relHFc}.

\begin{figure}[ht]
\begin{center}
\includegraphics[height=6cm,width=4cm]{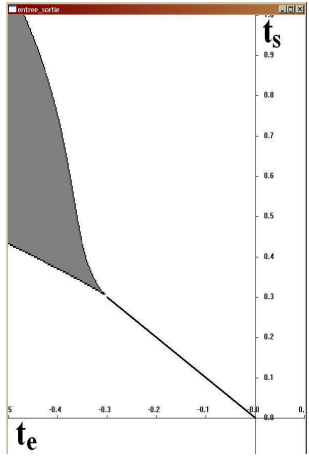}
\end{center}
\caption{The input-output relation for equation (\ref{eqx}) when there exists a big canard.}\lb{fig:relHFc}
\end{figure}

We could give more precise results if we consider the two variables $r$ and $\theta$ for the input-output relation. Indeed, when the point $(t_e,t_s)$ is in the interior of the graph of the input-output relation, we know that, at time of output, $\theta$ is going along the attractive slow curve which corresponds to the unique fast trajectory tangent to the eigenspace of the biggest eienvalue $\mu$.

\subsection{The focus-node bifurcation is a  bump}\lb{para:HFNb}
\setcounter{footnote}{0}

Here is the main part of this article. Today, I am not able to prove the expecting results, but I have propositions in this direction. To explain the problem, I will give conjectures.

Let us define the anti-bump $t^{**}$ and the two bumps $t^*_\lambda$ and $t^*_\mu$ as in definition \ref{def:butee}: 
$$R_\lambda(t_c)\ =\ R_\mu(t_c)\ =\ R_\lambda(t^{**})\ =\ R_\mu(t^{**})\ =\ R_\lambda(t^*_\lambda)\ =\ R_\mu(t^*_\mu)$$.
We have~~~~ $t^{**}<a<t^*_\mu= t^*_\lambda\leq b$~~ or~~ $t^{**}<a<b<t^*_\mu< t^*_\lambda$. In the first case, the bump is before the focus node bifurcation, and the study of paragraph \ref{Hopf-butée} is available. The interesting case is the second, where the computed bump is after the focus node bifurcation, this case is assumed with hypothesis HFN5.
\begin{conjecture}\lb{conj1} With hypothesis HFN1 to HFN5, the following proposition is generically wrong: 

If a trajectory of (\ref{eqX}) goes along the slow curve at least on $]t^{**},a[$, then it goes until the slow curve at least on $[t^{**},t^*_\mu]$. 
\end{conjecture}

To work on this conjecture, we will study an example which is, in some sense, a normal form of the problem: the slow curve is moved on the $t$-axis and the fast vector field is linearized. The example is
\systemedeux{eqx}{\eps^3\dot{x}}{tx+y+\eps^3c_1}{\eps^3\dot{y}}{(t-b)x+ty+\eps^3c_2}

\begin{proposition} A { numerical} simulation of equation  (\ref{eqx}) gave the figure  \ref{fig:traj}. It confirms conjecture \ref{conj1}.
\end{proposition}

\begin{figure}[ht]
\begin{center}
\includegraphics[height=4cm,width=7cm]{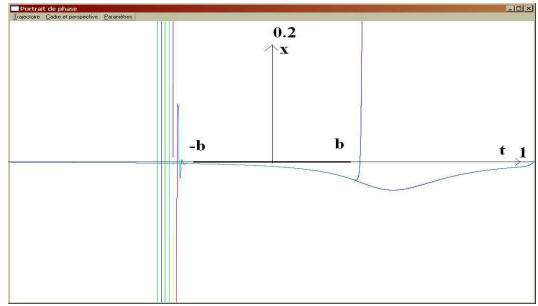}
\end{center}
\caption{Trajectories $X_-$ and $X_+$: the first goes along the horizontal axis from $-\infty$ to $b$, where it jumps outside the neighborhood of the horizontal axis; the second one goes along the horizontal axis from $+\infty$ to $-b$ where it has big oscillations. The parameters are $b=0.3$, $c_1=0$, $c_2=-1$, $\eps^3=0.002$, the trajectory is computed with a RK4 method, with step $0.0001$. Other methods and other steps were tried, and the results are always very similar. }\lb{fig:traj}
\end{figure}

This proposition gives a good argument for the next conjecture, more precise than the first one:

\begin{conjecture}\lb{conj2} If a trajectory of system  (\ref{eqX}) goes along the slow curve in a neighborhood of a real $t$ with $t<a$ and $R(t)>R(b)$, then it does not go along the slow curve after the focus-node bifurcation point  $b$.
\end{conjecture}
So, generically, the input-output relation of equation (\ref{eqX}) has a graph similar to the graph of figure  \ref{fig:relH}; if $R(t^{**})>R(b)$, we have to replace $t^*$ by $b$ et $t^{**}$ par $t^{**}_b$ where $R(t^{**}_b)=R(b)$. The delay of the Hopf bifurcation is stopped either by the bump (as in case of a Hopf bifurcation alone) either by the focus-node bifurcation.

\begin{proposition} If the conjecture \ref{conj2} is true for one trajectory, then it is true for all of them.
\end{proposition}

\begin{proof} Assume that equation (\ref{eqX}) has a solution $\tilde{X}$ which does not verify conjecture \ref{conj2}. Then, $\tilde{X}$ goes along the slow curve on an interval  $]t_1,t_2[$ with $t_1<t^{**}_b<a<b<t_2$. If the problem is considered on a restricted interval $]t_1,t_2[$, the equation has a big canard, and we can apply the  proposition \ref{prop:HFNcan}. Then all trajectories  going along the slow curve before $t^{**}_b$ goes along the slow curve until $b$, and even a little more.
\end{proof}

In this article, we will now study only equation  (\ref{eqx}). We changed $\eps$ into $\eps^3$ only to avoid fractionnary exponents. The analytic structure with respect to $\eps$ is obviously modified, but does not matter for our purpose.

To study the phase portrait of equation (\ref{eqX}) or (\ref{eqx}), two trajectories are very important. They are called {\it distinguished trajectories} by JL.Callot and they are very classical. The first one, denoted $X_+$ goes along the slow curve for $t$ near $t_M$. Similarly, $X_-$ goes along the slow curve for $t$ near $t_m$. These two trajectories are Fenichel's manifolds, they are unique when $t_m=-\infty$ and $t_M=+\infty$. For the particular equation (\ref{eqx}), these two trajectories are drawn on figure \ref{fig:traj}.  We have for this example a nice fact:  $X_-$ and $X_+$ have an explicit formula, using the Airy function (in an appendix (section  \ref{annexe}) , we give classical needed results on Airy functions and Airy equation). 
\begin{equation}\lb{solp}
X_+(t)\ =\ \left(\begin{array}{c}x_+(t)\\y_+(t)\end{array}\right)\ =\ 
-e^{\ud\frac{t^2}{\eps^3}}M(t)\int_t^{+\infty}e^{-\ud\frac{\tau^2}{\eps^3}}M^{-1}(\tau)d\tau
\left(\begin{array}{c}c_1\\c_2\end{array}\right)
\end{equation}\begin{equation}\lb{solm}
X_-(t)\ =\ \left(\begin{array}{c}x_-(t)\\y_-(t)\end{array}\right)\ =\ 
e^{\ud\frac{t^2}{\eps^3}}M(t)\int_{-\infty}^te^{-\ud\frac{\tau^2}{\eps^3}}M^{-1}(\tau)d\tau
\left(\begin{array}{c}c_1\\c_2\end{array}\right)
\end{equation}
\begin{equation}\lb{eqM}
\mbox{where~~}M(t)\mbox{~~$=$~~}\sqrt{\frac{\pi}{\eps}}\left(\begin{array}{cc}
A\left(j\frac{t-b}{\eps^2}\right)&A\left(j^2\frac{t-b}{\eps^2}\right)\\
\eps jA'\left(j\frac{t-b}{\eps^2}\right)&\eps j^2A'\left(j^2\frac{t-b}{\eps^2}\right)\end{array}\right) \mbox{~~~~with~~~~}
\det(M(t))\ =\ \frac{i}{2}
\end{equation}
All the integrals are convergent because the Airy function is bounded at infinity by $C|t|^{-\td}e^{\dt |t|^\td}$.
 
\subsubsection{The relief}\lb{par:relief}

In this paragraph, we want to explore the methods used in paragraph \ref{Hopf-butée} when there is a focus-node bifurcation. We also check the hypothesis HFN1 to HFN5.

Hypothesis HFN1 to HFN3 are obvious with the slow curve $\phi(t,X,0)=0$ and the domain  ${\cal D}=\CC\times\CC^2\times\CC$.

The computation of the eigenvalues of the jacobian matrix  
$J(t)=\left(\begin{array}{cc}t&1\\t-b&t\end{array}\right)$ 
gives
$$\lambda(t)=t-(t-b)^\ud\mbox{~~~~~~}\mu(t)=t+(t-b)^\ud$$

The determination of the square root is needed to allow the formula above. In all this paragraph, we choose a cut-off on the positive real axis:
$$(re^{i\theta})^\ud\:=\ \sqrt{r}\ e^{\frac{i\theta}{2}}\mbox{~~~~~~}\theta\in[0,2\pi[$$

For the function $()^\td$, we choose the same cut-off.

The relation $\overline{t^\ud}=-{\overline{t}}^\ud$ will be useful. Then, $\lambda$ and $\mu$ are the two determinations of a multiform function. The cut-off is the semi-axis $[b,+\infty[$, and $\overline{\mu(t)}=\lambda(\overline{t})$.

For $a=0$ and  
\begin{equation}\lb{hyp:b1}b>\uq,\end{equation}
the hypothesis HFN4 is easy to check.

The two associated reliefs are given by
$$F_\lambda(t)=\ud t^2-\dt(t-b)^\td-\dt ib^\td\mbox{~~~~~~}F_\mu(t)=\ud t^2+\dt(t-b)^\td+\dt ib^\td$$
$$R_\lambda(t)=\Re(F_\lambda(t))\mbox{~~~~~~~~~~~~~~~~}R_\mu(t)=\Re(F_\mu(t))$$

\begin{figure}[ht]
\begin{center}
\includegraphics[height=7cm,width=10cm]{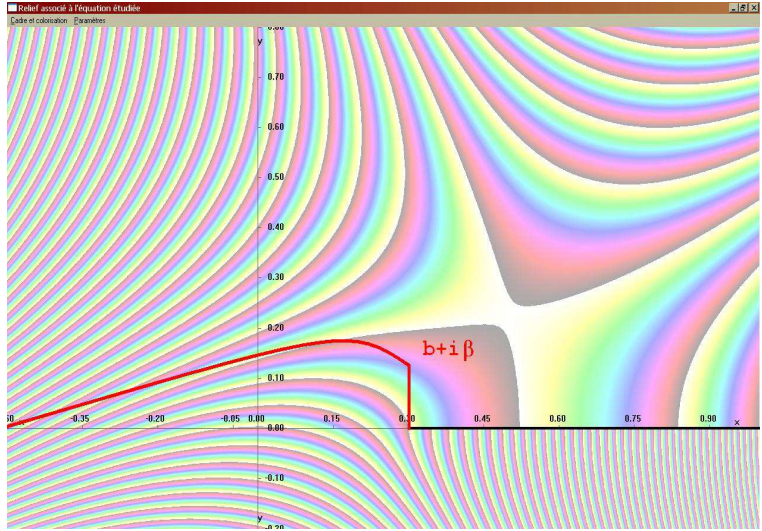}
\end{center}
\caption{Level curves of relief $R_\lambda$ for $b=0.3$, and path used in paragraph \ref{para:Xm}.}\lb{fig:relief}
\end{figure}

Let us comment figure \ref{fig:relief}: the value of $R_\lambda$ is $+\infty$ at both ends of the real axis. If a path goes from $t=-\infty$ to $t=+\infty$, it has to go down at least until the mountain pass, which is the unique critical point of the relief given by
$$t_c\ =\ \ud\ +i\sqrt{b-\uq}\mbox{~~~~~~}=\ 0.500+0.224\ i$$
The value of the relief at this critical point is
$$R_c=R_\lambda(t_c)\ =\ \ud b-\uud\mbox{~~~~~~}=\ 0.067$$
We solve now on the real axis the equation $R_\lambda(t)=R_c$. The solution are $t_e$ and $t_s$ given by
$$\left\{\begin{array}{ccccccc}
\mbox{~~if~~}&b>\ud+\us\sqrt{3}&~~,~~&t_e&=&-\sqrt{b-\us}&~~~~t_s=\sqrt{b-\us}\\
\mbox{~~if~~}&b<\ud+\us\sqrt{3}&~~,~~&t_e&=&-\sqrt{b-\us}&
\left\{\begin{array}{ccl}
t_{s1}&=& ...\\
t_{s2}&=& ...
\end{array}\right.
\end{array}\right.$$
The symbols $...$ in the formula above are the solutions of a polynom in $t$ of degree $4$. The exact expression is not needed. For $b=0.3$, we have
$$t_e=-0.365\mbox{~~~~~~}t_{s1}=0.346\mbox{~~~~~~}t_{s2}=0.525$$
The value $t_{s1}$ is on the sheet right to the cut-off: $\arg(t_{s1})=2\pi$. Besides, $t_{s2}$ is on the sheet left to the cut-off: $\arg(t_{s2})=0$. When we look on the polynom which has $t_{s1}$ and $t_{s2}$ as roots, we can prove that the hypothesis HFN5 is satisfied for 
\begin{equation}\lb{hyp:b}
\uq\ <\ b\ <\ \ud+\us\sqrt{3}
\end{equation}

\subsubsection{Callot's domains}
To study the canards of equation (\ref{eqx}), we introduce two special solutions, called distinguished solutions by J.L. Callot:  $X_+=(x_+,y_+)$ has an asymptotic\footnote{Here the things are easier than in the general case because the domain ${\cal D}_t$ contains the whole real axis. In general case, there is no unicity of the distinguished solution, but the difference remains exponentially smaller than the computed quantities.} condition  $X_+(+\infty)=0$ and  $X_-=(x_-,y_-)$ has an asymptotic condition $X_-(-\infty)=0$. They are unique. In this paragraph we build a domain ${\cal D}_+$ where $X_+$ is infinitesimal (it corresponds in the complex plane to the expression "going along a real interval"). In allmost all situations, the builded domain is the maximal domain with this property.

\paragraph{For trajectory $X_+$}
In this paragraph, it is better to change the cut-off, and we define (only in this paragraph)
$$(re^{i\theta})^\ud\:=\ \sqrt{r}\ e^{\frac{i\theta}{2}}\mbox{~~~~~~}\theta\in[-\ud\pi,\td\pi[$$
We are looking for a complex domain ${\cal D}_+$ such that the real point $+\infty$ is in ${\cal D}_+$, the singularity $b$ is not in ${\cal D}_+$. We look for domains below $+\infty$ (see definition \ref{def:domainedescendant}) for the relief $R_\lambda$  and also below $+\infty$ for the relief $R_\mu$. 

On figure \ref{fig:domainep}, such domain is drawn\footnote{The picture is a little bit different when $b$ is greater or smaller than $-\td-\us\sqrt{123}$. For this particular value, we have $R_\mu(b)=R_\mu(t_c)$.} in dark. Attention: at the left, the domain has a spike with a real part smaller than $-b$ and a nonzero imaginay part. The intersection of ${\cal D}_+$ with the real axis is  $]-b,b[\cup]b,+\infty[$. The theorem of Callot (theorem \ref{th:Callot}) says that  $X^+$ is infinitesimal on the whole $S$-interior of ${\cal D}_+$. 

Actually, a more precise study shows that the domain ${\cal D}_+$ is not the maximal domain where $X^+$ is infinitesimal: if we consider domains on the the Riemann surface (two sheets covering) we can add to ${\cal D}_+$ its conjugate (drawn in lightgray on the figure \ref{fig:domainep}). Because the solution $X^+$ is analytic without singularity at point $b$, it is infinitesimal on the symetric domain.

\begin{figure}[ht]
\begin{center}
\includegraphics[height=7cm,width=10cm]{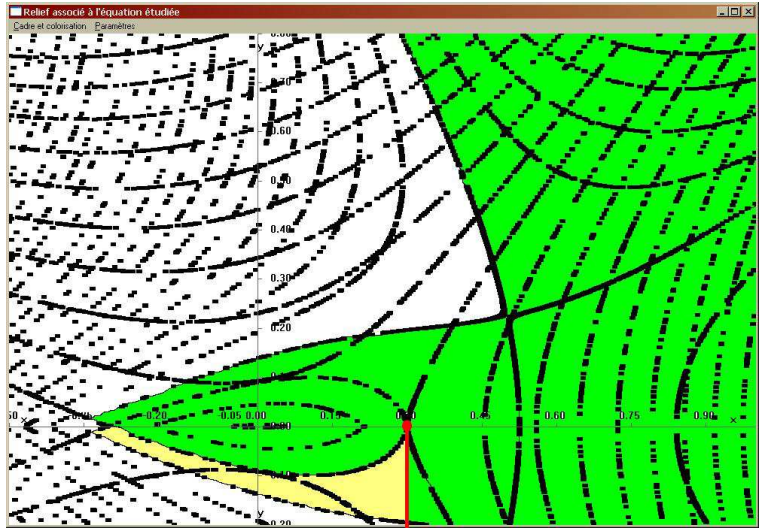}
\end{center}
\caption{The domain ${\cal D}_+$ for $b=0.3$}\lb{fig:domainep}
\end{figure}

\paragraph{For trajectory $X_-$}
A similar method gives the domain ${\cal D}_-$ such that $X_-$ is infinitesimal on the $S$-interior of ${\cal D}_-$. It is easier because we do not need to consider a two sheets covering. The domain ${\cal D}_-$ is drawn on figure  \ref{fig:domainem}.

\begin{figure}[ht]
\begin{center}
\includegraphics[height=7cm,width=10cm]{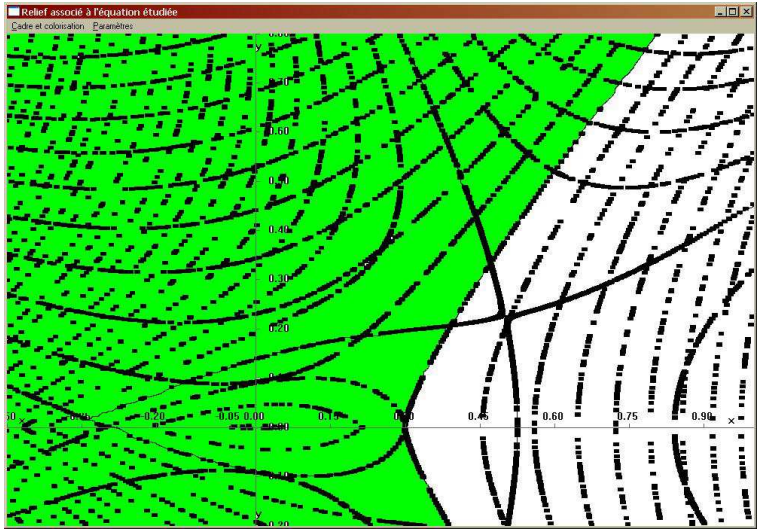}
\end{center}
\caption{The domain ${\cal D}_-$ for $b=0.3$}\lb{fig:domainem}
\end{figure}

\subsubsection{Evaluation of $X_+(b)$}\lb{para:Xp}
The slow curve $x=y=0$ is repulsive for all positive $t$. Then the trajectory $X_+$ is infinitesimal at least for all $t$ positive non infinitesimal (in fact it is infinitesimal on a larger interval). Its asymptotic expansion in power of $\eps^3$ is given by formal identification in the equation:  $X_+=\sum_{n\geq 0} X_n(t)\eps^{3n}$ has to verify the recurrence identities 
$$\systemededeux{\dot{x}_{n-1}}{tx_n+y_n+\delta_{n-1}c_1}{\dot{y}_{n-1}}{(t-b)x_n+ty_n+\delta_{n-1}c_2}$$
where $\delta_{n-1}=1$ if $n=1$ and vanishes for all others $n$. 
The computation of the first terms is easy:
$$x(t)\ =\ \frac{-c_1t+c_2}{t^2-t+b}\eps^3\ +\ \frac{t(t^2+t-3b)c_1+(-3t^2+t+b)c_2}{(t^2-t+b)^3}\eps^6\ +\ O(\eps^9)$$
$$y(t)\ =\ \frac{(t-b)c_1-tc_2}{t^2-t+b}\eps^3\ +\ \frac{(-2t^3+3bt^2+bt-b^2)c_1+(t^3+2t^2-3bt-t+b)c_2}{(t^2-t+b)^3}\eps^6\ +\ O(\eps^9)$$
and we we have now proved the

\begin{proposition}
$$X_+(b)\ =\ \left(\begin{array}{c}
\left(-\frac{1}{b}c_1+\frac{1}{b^2}c_2\right)\eps^3\ +\ \left(\left(\frac{1}{b^3}-\frac{2}{b^4}\right)c_1+\left(-\frac{3}{b^4}+\frac{2}{b^5}\right)c_2\right)\eps^6\ +\ O(\eps^9)\\
~\\
-\frac{1}{b}c_2\eps^3\ + \left(\frac{1}{b^3}c_1+\left(\frac{1}{b^3}-\frac{1}{b^4}\right)c_2\right)\eps^6\ +\ O(\eps^9)\\
\end{array}\right)$$
\end{proposition}

\subsubsection{Evaluation of $X_-(b)$}\lb{para:Xm}
The simple method above is not convenient to evaluate $X_-(b)$ because we expect that $X_-$ does not go along the slow manifold in a neighborhood of $b$. 

We will use the explicit formula  (\ref{solm}) to evaluate  $X_-(b)$. The computation is a little bit tedious. In all the formulae below, the symbol $\ip$~ represent a quantity which goes to zero when $\eps>0$ goes to zero.

The inverse of the matrix $M$ is easy to compute: we know the determinant of $M$ (see the property \ref{Airy:det} in the appendix on Airy's functions). 
$$M^{-1}(\tau)\ =\ -2i\sqrt{\frac{\pi}{\eps}}\left(\begin{array}{cc}
\eps j^2A'\left(j^2\frac{\tau-b}{\eps^2}\right)&-A\left(j^2\frac{\tau-b}{\eps^2}\right)\\
-\eps jA'\left(j\frac{\tau-b}{\eps^2}\right)&A\left(j\frac{\tau-b}{\eps^2}\right)\end{array}\right)$$

To compute the integrals in formula (\ref{solm}), we change the real path  of integration $]-\infty,b]$. For some integrals we choose a path which goes down the relief $R_\lambda$ from $-\infty$ to $b$, for other integrals, we choose the conjugate path which goes down the relief $R_\mu$ (the idea is the same as in Callot's proof of theorem \ref{th:Callot}). The path which goes down $R_\lambda$ is drawn on figure \ref{fig:relief}. The end of the path is a vertical segment from $b+i\beta$ to $b$. At point $b$, it is tangent to the level curve of the relief, then, the path  does not go down the relief with the precise definition \ref{def:chemindescendant}. Thus, we have to be care with approximations at this point.

Let us denote
$$f(\tau)\ =\ e^{\ud\frac{b^2-\tau^2}{\eps^3}}A\left(j^2\frac{\tau-b}{\eps^2}\right)$$
It is one of the function we have to integrate to evaluate $X_-$. 

\begin{lemma}\lb{lem:estim}
Let us give $\tau$ such that $\tau-b$ is non infinitesimal and 
$\ut\pi<\arg(\tau-b)<\pi$. Then
$$|f(\tau)|\ =\  \exp\left(\frac{-1}{\eps^3}(R_\lambda(\tau)-R_\lambda(b)+\ip)\right) $$ 
\end{lemma}
\begin{proof}
Using the asymptotic expansion of $A$ (see in appendix), we have:
$$A\left(j^2\frac{\tau-b}{\eps^2}\right)=\frac{1}{2\sqrt{\pi}}\exp\left(-\frac{2}{3}\left(j^2\frac{\tau-b}{\eps^2}\right)^{\td}\right)\left(j^2\frac{\tau-b}{\eps^2}\right)^{-\uq}(1+O(\eps^3))$$
Substituting in the formula of $f$, we have:
$$\eps^3\ln|f(\tau)|\ =\ \Re\left(\ud b^2-\ud\tau^2-\dt(j^2(\tau-b))^\td+\ip\right)$$
We write $\tau-b$ in polar coordinates: $\tau-b=re^{i\theta}$, with $\theta\in]\ut\pi,\pi[$. 
Then  $j^2(\tau-b)=re^{i(\theta-\dt\pi)}$. Because $\theta-\dt\pi$ has an argument between $-\ud\pi$ and $\ud\pi$, the power $\td$ gives $\left(j^2(\tau-b)\right)^{\td}=r^{\td}e^{i(\td\theta-\pi)}$. This expression can be writed $-(\tau-b)^\td$, with the same determination of  $t^{\td}$ as in  $R_\lambda$. 
\end{proof}
The interesting consequence of this lemma is that along the considered path, the function $f$ is increasing with a logarithmic derivative of type $\eps^{-3}$.
To precise, we need the following lemma:

\begin{lemma}\lb{lem:majo:f} 
There exist two constants $k$ and $\delta$ standard\footnote{here, it is the same to assume that $k$ and $\delta$ are independent of $\eps$}, positive such that
$$\forall\sigma\in[0,\frac{\beta}{\eps^2}]\mbox{~~~~,~~~~}\left|{f(b+i\sigma\eps^2)}\right|<ke^{-\delta\sigma^\td}$$ 
\end{lemma}
\begin{proof} By définition of $f$, we have 
$$f(b+i\sigma\eps^2)\ =\ e^{-\frac{b}{\eps}\sigma i}e^{\ud\sigma^2\eps}A(ij^2\sigma)\mbox{~~~~~~then~~~~~~}
\left|f(b+i\sigma\eps^2)\right|\ =\ e^{\ud\sigma^2\eps}\left|A(ij^2\sigma)\right|$$
For real positive infinitely large $\sigma$, the asymptotic expansion of the Airy function give the estimation
$$|A(ij^2\sigma)|\ =\ \frac{1}{2\sqrt{\pi}}\left|e^{-\dt (ij^2\sigma)^\td}\right||ij^2\sigma|^{-\uq}(1+\ip)\ =\ 
\frac{1}{2\sqrt{\pi}}\sigma^{-\uq}e^{-\frac{\sqrt{2}}{3}\sigma^\td}(1+\ip)$$
(we know that $\Re((ij^2)^\td)=\frac{\sqrt{2}}{2}$). Then if $\delta_1$ is real standard, less than  $\frac{\sqrt{2}}{3}$, we have the following inequality, true for all  $\sigma$ infinitely large:
$$|A(ij^2\sigma)|\ <\ e^{-\delta_1\sigma^\td}$$
By permanence\footnote{The non standard arguments in these proofs can be replaced by classical arguments, but, for that, new quantified variables have to be added, and it seems to me that the idea of the proof is more understandable with nonstandard language. }, this inequality believes true for all real $\sigma$ greater than some positive standard $\omega$. 
We can deduce the following majoration:
$$\forall\sigma\in[\omega,\frac{\beta}{\eps^2}]\mbox{~~~~,~~~~}\left|{f(b+i\sigma\eps^2)}\right|\ <\ e^{\ud\sigma^2\eps}e^{-\delta_1\sigma^{\td}}$$

For $\sigma<\omega$, we have:
$$\forall\sigma\in[0,\omega]\mbox{~~~~,~~~~}\left|{f(b+i\sigma\eps^2)}\right|\ <\ e^{\ud\omega^2\eps}k_1\ <\ 2k_1\mbox{~~~~with~~~~}k_1=\max_{\sigma\in[0,\omega]}|A(ij^2\sigma)|$$
Then we are looking for a constant  $\delta$ such that
$$\forall\sigma\in[0,\frac{\beta}{\eps^2}]\mbox{~~~~,~~~~}\ud\sigma^2\eps-\delta_1\sigma^\td\ <\ -\delta\sigma^\td$$
The inequality is equivalent to  $\sigma<\frac{4(\delta_1-\delta)^2}{\eps^2}$. A  choice of $\delta$  less than $\delta_1-\ud\sqrt{\beta}$ is convenient. This choice is possible only if $\delta_1>\ud\sqrt{\beta}$ what is true as soon as  $\beta<\frac{8}{9}$ and $\delta_1$ near enough from  $\frac{\sqrt{2}}{3}$. 

To verify the majoration of the lemma for $\sigma<\omega$, we can choose 
$$k\ =\ 2k_1e^{\delta\omega^\td}$$. 
\end{proof}

The next lemma is the more technical part of the article. The purpose is to evaluate an oscillating integral with successive integrations by parts.

\begin{lemma}\lb{lem:boutduchemin} 
We have the following expansion:
$$\int_{b+i\beta}^b f(\tau)d\tau\ =\ -\frac{1}{b}A(0)\eps^3\ -\ \frac{j^2}{b^2}A'(0)\eps^4\ -\ \frac{j}{b^3}A''(0)\eps^5\ +\ \left(\frac{1}{b^3}A(0)-\frac{1}{b^4}A'''(0)\right)\eps^6\ +\ \left(\frac{3j^2}{b^4}A'(0)-\frac{j^2}{b^5}A''''(0)\right)\eps^7\ +\ \ip\eps^7$$

\end{lemma}
\begin{proof}
Let us substitute $\tau$ by $b+i\eps^2\sigma$ in the integral. We have 
$$\int_{b+i\beta}^b f(\tau)d\tau \ =\ 
-i\eps^2\int_0^{\frac{\beta}{\eps^2}}f(b+i\eps^2\sigma)d\sigma\ =\ 
-i\eps^2\int_0^{\frac{\beta}{\eps^2}}e^{-\frac{b}{\eps}\sigma i}e^{\ud\sigma^2\eps}A(ij^2\sigma)d\sigma$$
The exponential $e^{-\frac{b}{\eps}\sigma i}$ is fast oscillating. The exponential $e^{\ud\sigma^2\eps}$ is infinitely close to  $1$ for all non infinitely large $\sigma$ and  $A(ij^2\sigma)$ is decreasing. All properties are checked to apply the method of integrations by parts. But there is a difficulty:  $e^{\ud\sigma^2\eps}$ is increasing and does not believe close to $1$. Now, let us explain the computations:

$$I\ =\ \int_{b+i\beta}^b f(\tau)d\tau\ =\ I_1+I_2+I_3\mbox{~~~~~~with}$$
$$I_1\ =\ \frac{\eps^3}{b}\left(f(b+\beta i)-f(b)\right)\mbox{~~~~~~} 
I_2\ =\ -\frac{\eps^4}{b}\int_0^{\frac{\beta}{\eps^2}}\sigma f(b+i\eps^2\sigma)d\sigma$$
$$I_3\ =\ -\frac{ij^2\eps^3}{b}\int_0^{\frac{\beta}{\eps^2}} \hat{f}(b+i\eps^2\sigma)d\sigma\mbox{~~~~with~~~~}
\hat{f}(b+i\eps^2\sigma)\ =\ e^{-\frac{b}{\eps}\sigma i}e^{\ud\sigma^2\eps}A'(ij^2\sigma)$$
With lemma \ref{lem:majo:f}, we know that  $f(b+\beta i)$ is exponentially smaller than  $f(b)=A(0)$.  Thus, we have 
$$I_1\ =\ 
-\frac{ 1}{b}A(0)\eps^3+\ip\eps^7$$
To estimate $I_2$, we perform a new integration by parts: 
$$I_2\ =\ J_1+J_2+J_3+J_4\mbox{~~~~~~with}$$
$$J_1\ =\ -\frac{i\eps^5}{b^2}\frac{\beta}{\eps^2}f(b+\beta i)\mbox{~~~~~~} 
J_2\ =\ \frac{i\eps^6}{b^2}\int_0^{\frac{\beta}{\eps^2}}\sigma^2 f(b+i\eps^2\sigma)d\sigma$$
$$J_3\ =\ -\frac{j^2\eps^5}{b^2}\int_0^{\frac{\beta}{\eps^2}}\sigma \hat{f}(b+i\eps^2\sigma)d\sigma\mbox{~~~~~~}
J_4\ =\ \frac{i\eps^5}{b^2}\int_0^{\frac{\beta}{\eps^2}}f(b+i\eps^2\sigma)d\sigma$$
Because $f(b+\beta i)$ is exponentially small, we have $J_1=\ip\eps^7$. We have also  $J_4=-\frac{\eps^3}{b^2}I$. If you substitute $A'$ for $A$ the expression $I_3$ is the same as $\frac{j^2\eps}{b}I$. All the arguments are the same with function $A'$ and function $A$.  Let us denote  $\hat{I}_i$, $\hat{J}_i$ the expressions obtained from $I_i$ and $J_i$ when $A'$ is substituted for $A$. Thus we have  $I_3=\frac{j^2\eps}{b}\hat{I}$. To estimate $J_2$ we perform a new integration by parts exactly as for evaluation of $I_2$~: $J_2=K_1+K_2+K_3+K_4$. All the integrals are bounded by a non infinitely large real number because all the integrated functions are bounded  (see lemma \ref{lem:majo:f}) by a integrable standard function. To summarize:
$$I\ =\ I_1+I_2+I_3\mbox{~~~~}I_2=J_1+J_2+J_3+J_4\mbox{~~~~}J_2=K_1+K_2+K_3+K_4$$
$$I_1\ =\ -\frac{ 1}{b}A(0)\eps^3+\ip\eps^7\mbox{~~~~~~}
J_1\ =\ \ip\eps^7\mbox{~~~~~~}K_1\ =\ \ip\eps^7$$
$$I_2=\ip\eps^3\mbox{~~~~~~}J_2=\ip\eps^5\mbox{~~~~~~}
K_2=\frac{\eps^8}{b^3}\int_0^{\frac{\beta}{\eps^2}}\sigma^3 f(b+i\eps^2\sigma)d\sigma=\ip\eps^7$$
\begin{equation}\lb{formules}I_3=\frac{j^2\eps}{b}\hat{I}\mbox{~~~~~~}J_3=\frac{j^2\eps}{b}\hat{I}_2\mbox{~~~~~~}K_3=\frac{j^2\eps}{b}\hat{J}_2
\mbox{~~~~~~}J_4=-\frac{\eps^3}{b^2}I\mbox{~~~~~~}K_4=-\frac{2\eps^3}{b^2}I_2\end{equation}
Then, all the ingredients are given, and we can compute the asymptotic expansion of $I$ in powers of $\eps$. To start, we have  $I=\ip\eps$. For similar reason, $\hat{I}=\ip\eps$. Then, using formulae \ref{formules}, we have $I_3=\ip\eps^2$, then $I=\ip\eps^2$. We iterate the process, inserting the known approximations in formulae \ref{formules}, and we obtain a better approximation: $I_3=\ip\eps^3$ then 
$$I\ =\ -\frac{1}{b}A(0)\eps^3\ +\ \ip\eps^3$$
The next step:
$$I_3=-\frac{j^2}{b^2}A'(0)\eps^4+\ip\eps^4\mbox{~~~~~~~~}
J_3=\ip\eps^4\mbox{~~~~~~}J_4=\frac{1}{b^2}A(0)\eps^6+\ip\eps^6\mbox{~~~~~~}
I_2=\ip\eps^4$$
$$I\ =\ -\frac{1}{b}A(0)\eps^3\ -\frac{j^2}{b^2}A'(0)\eps^4\ +\ \ip\eps^4$$
The next step: (do not use the relation $A''(0)=0$, because we have sometimes to substitute $A'$ for $A$):
$$I_3=-\frac{j^2}{b^2}A'(0)\eps^4-\frac{j}{b^3}A''(0)\eps^5+\ip\eps^5\mbox{~~~~~~}J_3=\ip\eps^5\mbox{~~~~~~}I_2=\ip\eps^5$$
$$I\ =\ -\frac{1}{b}A(0)\eps^3\ -\frac{j^2}{b^2}A'(0)\eps^4\ -\ \frac{j}{b^3}A''(0)\eps^5\ +\ \ip\eps^5$$
The next step:
$$I_3=-\frac{j^2}{b^2}A'(0)\eps^4-\frac{j}{b^3}A''(0)\eps^5-\frac{1}{b^4}A'''(0)\eps^6+\ip\eps^6\mbox{~~~~~~}J_3=\ip\eps^6\mbox{~~~~~~}K_3=\ip\eps^6$$
$$J_4=\frac{1}{b^3}A(0)\eps^6+\frac{j^2}{b^4}A'(0)\eps^7+\ip\eps^7\mbox{~~~~~~}K_4=\ip\eps^8$$
$$J_2=\ip\eps^6\mbox{~~~~~~}I_2=\frac{1}{b^3}A(0)\eps^6+\ip\eps^6$$
$$I\ =\ -\frac{1}{b}A(0)\eps^3\ -\frac{j^2}{b^2}A'(0)\eps^4\ -\ \frac{j}{b^3}A''(0)\eps^5\ +\ \left(\frac{1}{b^3}A(0)-\frac{1}{b^4}A'''(0)\right)\eps^6+\ip\eps^6$$
The last step:
$$I_3=-\frac{j^2}{b^2}A'(0)\eps^4-\frac{j}{b^3}A''(0)\eps^5-\frac{1}{b^4}A'''(0)\eps^6+\left(\frac{j^2}{b^4}A'(0)-\frac{j^2}{b^5}A''''(0)\right)\eps^7+\ip\eps^7\mbox{~~~~~~}
J_3=\frac{j^2}{b^4}A'(0)\eps^7+\ip\eps^7\mbox{~~~~~~}
K_3=\ip\eps^7$$
$$J_2=\ip\eps^7\mbox{~~~~~~}I_2=\frac{1}{b^3}A(0)\eps^6+\frac{2j^2}{b^4}A'(0)\eps^7+\ip\eps^7$$
$$I\ =\ -\frac{1}{b}A(0)\eps^3\ -\ \frac{j^2}{b^2}A'(0)\eps^4\ -\ \frac{j}{b^3}A''(0)\eps^5\ +\ \left(\frac{1}{b^3}A(0)-\frac{1}{b^4}A'''(0)\right)\eps^6\ +\ \left(\frac{3j^2}{b^4}A'(0)-\frac{j^2}{b^5}A''''(0)\right)\eps^7\ +\ \ip\eps^7$$
\end{proof}

\begin{lemma}\lb{lem:debutduchemin} 
$$\int_{-\infty}^{b+i\beta} f(\tau)d\tau =\ \exp\left(\frac{-1}{\eps^3}(R_\lambda(b+i\beta)-R_\lambda(b)+\ip)\right)$$
\end{lemma}
\begin{proof} The chosen path goes down the relief $R_\lambda$, then the lemma  is a corollary of the majoration of lemma \ref{lem:estim}. 
\end{proof}

\begin{lemma}\lb{lem:DA}
$$\int_{-\infty}^{b}e^{\ud\frac{b^2-\tau^2}{\eps^3}}A\left(j^2\frac{\tau-b}{\eps^2}\right)d\tau\ =\ -\frac{1}{b}A(0)\eps^3\ -\ \frac{j^2}{b^2}A'(0)\eps^4\ +\ \left(\frac{1}{b^3}-\frac{1}{b^4}\right)A(0)\eps^6\ +\ \left(\frac{3j^2}{b^4}-\frac{2j^2}{b^5}\right)A'(0)\eps^7\ +\ \ip\eps^7$$
$$\int_{-\infty}^{b}e^{\ud\frac{b^2-\tau^2}{\eps^3}}A'\left(j^2\frac{\tau-b}{\eps^2}\right)d\tau\ =\ 
-\frac{1}{b}A'(0)\eps^3\ -\ \frac{j}{b^3}A(0)\eps^5\ +\ 
\left(\frac{1}{b^3}-\frac{2}{b^4}\right)A'(0)\eps^6\ +\ \ip\eps^7$$
$$\int_{-\infty}^{b}e^{\ud\frac{b^2-\tau^2}{\eps^3}}A\left(j\frac{\tau-b}{\eps^2}\right)d\tau\ =\ -\frac{1}{b}A(0)\eps^3\ -\ \frac{j}{b^2}A'(0)\eps^4\ +\ \left(\frac{1}{b^3}-\frac{1}{b^4}\right)A(0)\eps^6\ +\ \left(\frac{3j}{b^4}-\frac{2j}{b^5}\right)A'(0)\eps^7\ +\ \ip\eps^7$$
$$\int_{-\infty}^{b}e^{\ud\frac{b^2-\tau^2}{\eps^3}}A'\left(j\frac{\tau-b}{\eps^2}\right)d\tau\ =\ 
-\frac{1}{b}A'(0)\eps^3\ -\ \frac{j^2}{b^3}A(0)\eps^5\ +\ 
\left(\frac{1}{b^3}-\frac{2}{b^4}\right)A'(0)\eps^6\ +\ \ip\eps^7$$
\end{lemma}

\begin{proof}
With lemmas \ref{lem:boutduchemin} and \ref{lem:debutduchemin}, the firs part of the lemma is proved. A similar computation gives the second part, if we remember that $A''(0)=0$ but  $A'''(0)\neq 0$. So, the vanishing terms are not the same in the two formulas. The two last formulas are the complex conjugate of the two first one.\end{proof}

\begin{proposition}
$$X_-(b)\ =\ \left(\begin{array}{c}
\left(-\frac{1}{b}c_1+\frac{1}{b^2}c_2\right)\eps^3\ +\ \left(\left(\frac{1}{b^3}-\frac{2}{b^4}\right)c_1+\left(-\frac{3}{b^4}+\frac{2}{b^5}\right)c_2\right)\eps^6\ +\ O(\eps^9)\\
~\\
-\frac{1}{b}c_2\eps^3\ + \left(\frac{1}{b^3}c_1+\left(\frac{1}{b^3}-\frac{1}{b^4}\right)c_2\right)\eps^6\ +\ O(\eps^9)\\
\end{array}\right)$$
\end{proposition}

\begin{proof}Insert the estimations of lemma \ref{lem:DA} in the explicit formula (\ref{solm}), and, after tedious simplifications, the proposition is proved.
\end{proof}

\begin{conjecture}
The two values $X_-(b)$ and $X_+(b)$ have the same asymptotic expansion. 
\end{conjecture}

With Maple, I checked that the two expansions coincide until terms in $\eps^9$.

\section{Appendix: Airy's functions}\lb{annexe}
The Airy's equation is linear, non autonomous of second order. It is
\begin{equation}\lb{eq:Airy}
\frac{d^2x}{dt^2}\ = \ tx
\end{equation}
The pair $(A(t),B(t))$ of Airy's functions is a fondamental system of solutions. The function satisfy the following properties (these results can be found in every book on special functions). 
\begin{enumerate}
\item The value at the origin are: $$A(0)=3^{-\dt}\frac{1}{\Gamma(\dt)}\mbox{~~~~~~}A'(0)=-\frac{3^\us}{2}\frac{\Gamma(\dt)}{\pi}\mbox{~~~~~~}
B(0)=3^{-\us}\frac{1}{\Gamma(\dt)}\mbox{~~~~~~}B'(0)=\frac{3^\dt}{2}\frac{\Gamma(\dt)}{\pi}$$

\item \lb{Airy:DA} On a sector of angle less than $\dt\pi$, around the positive real axis\footnote{Take care: the determination of $t^\td$ is here the classical determination with a cut off on the negative real axis, not the determination choose along all this article}, the Airy's functions have an asymptotic expansion for $t$ going to infinity:  
$$A(t)\ =\frac{1}{2\sqrt{\pi}}{e}^{-\dt\,{t}^{\td}}{t}^{-\uq}(1+O(t^{-\td}))\mbox{~~~~~~~~}
A'(t)\ =-\frac{1}{2\sqrt{\pi}}{e}^{-\dt\,{t}^{\td}}{t}^{\uq}(1+O(t^{-\td}))$$
$$B(t)\ =\frac{1}{\sqrt{\pi}}{e}^{\dt\,{t}^{\td}}{t}^{-\uq}(1+O(t^{-\td}))\mbox{~~~~~~~~}
B'(t)\ =\frac{1}{\sqrt{\pi}}{e}^{\dt\,{t}^{\td}}{t}^{\uq}(1+O(t^{-\td}))$$
The functions $A$ et $B$ are oscillating when $t$ goes to $-\infty$.

\item Let us denote $j=e^{\dt i\pi}=-\ud+\frac{\sqrt{3}}{2}i$. The Airy's equation is invariant by the change of variable  $t\mapsto jt$, then  $A(jt)$ and $B(jt)$ are also solutions. So they can be written as a linear combination of $A(t)$ and $B(t)$. We perform an identification at point $0$ to find the coefficients:
$$A(jt)\ = -\ud j^2A(t)+\ud ij^2 B(t)\mbox{~~~~~~}B(jt)\ =\ \td ij^2A(t)-\ud j^2B(t)$$
$$A(j^2t)\ =\  -\ud jA(t)-\ud ij B(t)\mbox{~~~~~~}B(j^2t)\ =\ -\td ijA(t)-\ud jB(t)$$

\item \lb{Airy:det} Classicaly, the couple $(A(t),B(t))$ is chosen for a base of the set of solutions. It could be better (in a study in the complex plane) to choose $(A(jt),A(j^2t))$ for base. With Liouville's theorem, we prove that the following determinant is constant, and we compute its value at the origin.
$$\det\left(\begin{array}{cc}
A(jt)&A(j^2t)\\jA'(jt)&j^2A'(j^2t)
\end{array}\right)\ =\ \frac{i}{2\pi}$$
\end{enumerate}

\begin{figure}[ht]\lb{fig:Airy}
\begin{center}
\includegraphics[height=4cm,width=4cm]{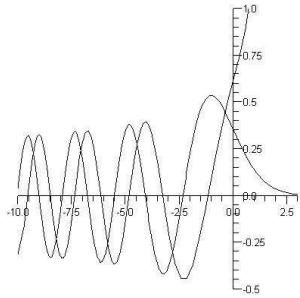}
\end{center}
\caption{Graphs of Airy's real functions $A$ and $B$}
\end{figure}


\bibliographystyle{plain}
\bibliography{general}

\end{document}